\documentclass[12pt,a4paper]{article}

\usepackage[utf8]{inputenc}
\usepackage{amsmath, amssymb, amsthm, amsfonts, mathrsfs}
\usepackage{bbm, commath}
\usepackage{multicol}
\usepackage{graphicx}
\usepackage{tikz-network}
\usetikzlibrary {positioning,patterns, calc}
\usetikzlibrary{graphs,graphs.standard,quotes}
\usepackage{authblk}
\usepackage{enumitem}

\newtheorem{thm}{Theorem}

\newtheorem{prop}{Proposition}
\newtheorem{lem}{Lemma}
\newtheorem{cor}{Corollary}

\def \Nl {{\mathbbm N}}
\def \Zl {{\mathbbm Z}}
\def \Ql {{\mathbbm Q}}
\def \Rl {{\mathbbm R}}
\def \Cl {{\mathbbm C}}
\def \e {{\bf e}}
\def \vl {{\bf v}}
\def \J {{\mathbb J}}

\newcommand{\ob}[1]{\left(#1\right)}
\newcommand{\cb}[1]{\left\lbrace #1\right\rbrace}
\newcommand{\tb}[1]{\left[#1\right]}

\title{\it Laplacian State Transfer on Graphs with an Edge Perturbation Between Twin Vertices}
\author[1]{Hiranmoy Pal\footnote{E-mail: palh@nitrkl.ac.in}}
\affil[1]{National Institute of Technology Rourkela, India-769008.}
\date{\today}

\begin{document}

\maketitle


\begin{abstract}
We consider quantum state transfer relative to the Laplacian matrix of a graph. Let $N(u)$ denote the set of all neighbors of a vertex $u$ in a graph $G$. A pair of vertices $u$ and $v$ are called twin vertices of $G$ provided $N\ob{u}\setminus\cb{v}=N\ob{v}\setminus\cb{u}$. We investigate the existence of quantum state transfer between a pair of twin vertices in a graph when the edge between the vertices is perturbed. We find that removal of any set of pairwise non-adjacent edges from a complete graph with a number of vertices divisible by $4$ results Laplacian perfect state transfer (or LPST) at $\frac{\pi}{2}$ between the end vertices of every edge removed. Further, we show that all Laplacian integral graphs with a pair of twin vertices exhibit LPST when the edge between the vertices is perturbed. In contrast, we conclude that LPST can be achieved in every complete graph between the end vertices of any number of suitably perturbed non-adjacent edges. The results are further generalized to obtain a family of edge perturbed circulant graphs exhibiting Laplacian pretty good state transfer (or LPGST) between twin vertices. A subfamily of which is also identified to admit LPST at $\frac{\pi}{2}$.\\


\end{abstract}


\section{Introduction}
Continuous-time quantum walks, initially used by Farhi and Gutmann in \cite{farhi}, plays an important role in analysing various quantum transportation phenomena. Quantum state transfer is one such phenomenon where the states of physical systems are transferred between two points in a quantum-network. Several models are studied for efficient information processing, a good account of such studies on quantum-network engineering can be found in \cite{nik}. Let there be a quantum-network modelled by a graph $G$ with vertex set $\cb{u_1,u_2,\ldots,u_n}$. The adjacency matrix $A=\tb{a_{ij}}$ relative to $G$ is an $n\times n$ matrix with $a_{ij}=1$ if there is an edge between $u_i$ and $u_j$, otherwise $a_{ij}=0$. The degree matrix $D$ of $G$ is an $n\times n$ diagonal matrix indexed by the usual ordering of vertices, where the diagonal entries are the degree of the corresponding vertices. The positive semidefinite matrix $L=D-A$ is known as the Laplacian matrix of $G$ (see \cite{god0}). Unless otherwise stated, we consider all graphs to be simple, undirected and connected. A continuous-time quantum walk on a graph $G$ relative to the Laplacian matrix $L$ is defined by the unitary matrix
\begin{eqnarray}\label{eq00}
U_L(t):=\exp{\ob{-itL}}=\sum\limits_{k\geq 0}\frac{\ob{-it}^k}{k!}L^k,\text{ where } t\in\Rl \text{ and } i=\sqrt{-1}.
\end{eqnarray}
The transition matrix $U_L(t)$ describes the evolution of states in a quantum system in which the XYZ-interaction model (see \cite{bose1}) is adopted. However, in an XY-interaction model, the adjacency matrix of $G$ is considered in \eqref{eq00} instead of the Laplacian matrix when defining the continuous-time quantum walk. It is well known that for a regular graph the study of state transfer relative to adjacency matrix and Laplacian matrix are equivalent in a sense that both considerations provide the same information. Now we discuss few properties of the transition matrix. A graph $G$ is said to exhibit Laplacian perfect state transfer (LPST) between two distinct vertices $u_a$ and $u_b$ if there exists $\tau\in\Rl$ such that
\begin{equation}\label{e2}
U_L\ob{\tau}\e_a=\gamma\e_b,~\text{for some }\gamma\in\Cl,
\end{equation}
where $\e_a$ and $\e_b$ are the characteristic vectors corresponding to $u_a$ and $u_b$, respectively. The study of perfect state transfer in quantum spin networks was initiated by Bose \cite{bose}. We find that LPST occurs on several regular graphs, such as cubelike graphs \cite{ber, che}, integral circulant graphs \cite{mil4} and distance regular graphs \cite{cou2}, etc. The effect of Laplacian state transfer on few graph operations can be found in \cite{ack, alv, wang}. However, it is found in \cite{cou0} that there is no tree with more than two vertices exhibiting LPST. Remarkably, in \cite{krik1}, Kirkland et al. shown that the Laplacian quantum walk helps detect a faulty link or matchings in a complete graph. Few other results on LPST can be found in \cite{li,liu}. \par
In \eqref{e2}, if we have $a=b$ and $\tau\neq 0$ then $G$ is said to be periodic at the vertex $u_a$ at time $\tau$. A graph is called periodic if it is periodic at all vertices at the same time. The spectral decomposition (see \cite{horn}) of the transition matrix can be used to find an important class of periodic graphs. Suppose the Laplacian matrix $L$ has the distinct eigenvalues $\mu_1, \mu_2, \ldots, \mu_k$. If the corresponding projections onto the eigenspaces of $L$ are denoted by $E_1, E_2, \ldots, E_k$, then the spectral decomposition of the transition matrix can be obtained as
\begin{eqnarray}\label{eq01}
U_L(t)=\exp{\ob{-itL}}=\sum\limits_{j=1}^{k}\exp{\ob{-i\mu_j t}}E_j.
\end{eqnarray}
It is now evident that if all Laplacian eigenvalues of a graph $G$ are integers then $G$ is periodic at $2\pi$ (see \cite{god3}). Such graphs with integer eigenvalues are known as Laplacian integral graphs. An important family of Laplacian integral graphs are the complete graphs. A complete graph on $n$ vertices is denoted by $K_n$ where every pair of vertices are joined by an edge. We denote $I$ to be the identity matrix and $J$ to be the matrix with all entries equal to $1$, where the size of the matrices $I$ and $J$ are assumed to be clear from the context. The Laplacian matrix of $K_n$ is obtained by $L=nI-J$ which has the eigenvalues $0$ and $n$ with multiplicities $1$ and $n-1$, respectively. The corresponding projections onto the eigenspaces are $\frac{1}{n}J$ and $I-\frac{1}{n}J$. Therefore the spectral decomposition of the transition matrix of $K_n$ relative to the Laplacian can be evaluated as
\begin{eqnarray}\label{eq0}
U_L(t)=\frac{1}{n}J + \exp{\ob{-int}} \ob{I-\frac{1}{n}J}.
\end{eqnarray}  
If $u_a$ and $u_b$ are two distinct vertices of $K_n$ then
\[\abs{\e_b^T U_L\ob{t}\e_a} \leq \frac{2}{n},\] 
which implies that all complete graphs with more than three vertices never admit LPST. Similar arguments  has been presented in \cite{coug} to establish that $K_n,~n\geq 3$, does not exhibit perfect state transfer relative to the adjacency matrix. In contrast, the main conclusion in \cite{bose1} observes that the complete graph $K_{4n}$ with a missing edge exhibits LPST. This motivates us to study state transfer on graphs with an edge perturbation between twin vertices that attributes to Laplacian perfect state transfer.\par
We also find a class of edge perturbed circulant graphs exhibiting Laplacian pretty good state transfer which was introduced in \cite{god1}. A graph $G$ is said to exhibit Laplacian pretty good state transfer (LPGST) between two distinct vertices $u_a$ and $u_b$ if there is a sequence $\tau_k\in\Rl$ such that
\begin{equation}\label{eqc1}
\lim\limits_{k\to\infty}U_L\ob{\tau_k}\e_a=\gamma\e_b,~\text{for some }\gamma\in\Cl.
\end{equation}
In \eqref{eqc1}, if we have $a=b$ and $\tau_k \neq 0$ for all $k$, then $G$ is said to be almost periodic at $u_a$ with respect to the sequence $\tau_k$. The graph is said to be almost periodic if there is a sequence $\tau_k (\neq 0)\in\Rl$ such that
\[\lim\limits_{k\to\infty}U_L\ob{\tau_k}=\gamma I,~\text{for some }\gamma\in\Cl,\]
where $I$ is the identity matrix of appropriate order. Among regular graphs, we find a good number of circulant graphs graphs exhibiting pretty good state transfer and almost periodicity in \cite{pal6,pal7,pal4}. Now we briefly define circulant graphs and introduce some relevant notations for convenience. Let $\left(\Gamma,+\right)$ be a finite abelian group. A Cayley graph over $\Gamma$ with the connecting set $S$ satisfying $0\notin S\subseteq\Gamma$ and $\left\lbrace -s:s\in S\right\rbrace=S$ is denoted by $Cay\left(\Gamma,S\right)$. The entries of $\Gamma$ are the vertices of $Cay\left(\Gamma,S\right)$ where two vertices $a,b\in\Gamma$ are adjacent if and only if $a-b\in S$. In case $\Gamma=\Zl_n$, then the Cayley graph is known as circulant graph. Let $n\in\Nl$ and $D$ be a set consisting of proper divisors of $n$ then for $d\in D$, we define
\[S_n(d)=\left\lbrace x\in\Zl_n: gcd(x,n)=d\right\rbrace \text{ and } S_n(D)=\bigcup\limits_{d\in D} S_n(d).\]
The set $S_n(D)$ is called a gcd-set of $\Zl_n$. In \cite{so}, we find a complete characterization of integral circulant graphs as follows.
\begin{thm}\label{so}
A circulant graph $Cay\left(\Zl_n,S\right)$ is integral if and only if $S$ is a gcd-set.
\end{thm}
The cycle $C_n$ is a Cayley graph over $\Zl_n$ with $S=\left\lbrace1,n-1\right\rbrace$. The eigenvalues and eigenvectors of $C_n$ are well known. Suppose $\omega_n=\exp{\left(\frac{2\pi i}{n}\right)}$ is the primitive $n$-th root of unity. For $0\leq l\leq n-1$, let $\lambda_l$ be the eigenvalue of $C_n$ corresponding to the eigenvector $\vl_l.$ Then 
\begin{eqnarray}\label{IE3}
\lambda_l = 2\cos{\left(\frac{2l\pi}{n}\right)},\text{ and } \vl_l = \left[1,\omega_n^l,\ldots,\omega_n^{l(n-1)}\right]^T.
\end{eqnarray}
Now we give a brief introduction to Kronecker approximation theorem on simultaneous approximation of numbers, which plays an important role in characterizing state transfer on circulant graphs. 
\begin{thm}\cite{apo}[Kronecker approximation theorem]
If $\alpha_1,\ldots,\alpha_l$ are arbitrary real numbers and if $1,\theta_1,\ldots, \theta_l$ are real, algebraic numbers linearly independent over $\Ql$ then for $\epsilon>0$ there exist $q\in\Zl$ and $p_1,\ldots,p_l\in\Zl$ such that
\[\left|q\theta_j-p_j-\alpha_j\right|<\epsilon.\]
\end{thm}

In the following sections we investigate Lpalacian state transfer in edge perturbed graphs.

\section{LPST on Edge Perturbed Graphs}

Let $G$ be a graph with vertex set $\cb{u_1,u_2,\ldots,u_n}$, and let the Laplacian matrix of $G$ be $L$. If $u_a$ and $u_b$ are two distinct non-adjacent vertices of $G$ then Laplacian matrix of $G+\cb{u_a,u_b}$, where $u_a$ and $u_b$ are joined by a new edge, becomes
\[L^{+}=L+\ob{\e_a-\e_b}\ob{\e_a-\e_b}^T,\] 
where $\e_a$ and $\e_b$ are the characteristic vectors corresponding to $u_a$ and $u_b$, respectively. Again, if $u_a$ and $u_b$ are two adjacent vertices of $G$ then the Laplacian matrix of the edge perturbed graph $G-\cb{u_a,u_b}$ is
\[L^{-}=L-\ob{\e_a-\e_b}\ob{\e_a-\e_b}^T.\]
Observe that $L^{+}$ and $L^{-}$ are perturbations of the Laplacian matrix $L$ with the rank one matrix $M=\ob{\e_a-\e_b}\ob{\e_a-\e_b}^T$. Here we consider a more general perturbation
\[L^{\alpha}=L+\alpha M,~\alpha\in\Rl.\]
The matrix $L^{\alpha}$ can be realised as the Laplacian matrix of $G+\alpha\cb{u_a,u_b}$ where the weight of the edge between $u_a$ and $u_b$ in $G$ is increased by $\alpha$.
Let $N(u)$ denote the set of all neighbors of a vertex $u$ in $G$. We prove the following result mentioned in \cite{so1}. Notably, the condition in the following Lemma guarantees the existence of an automorphism of the graph $G$ swapping $u_a$, $u_b$ and fixing other vertices.
\begin{lem}\label{lm1}
If $N\ob{u_a}\setminus\cb{u_b}=N\ob{u_b}\setminus\cb{u_a}$ then the matrices $L$ and $M$ commute.
\end{lem}   

\begin{proof} Enough to consider the following cases as the problem is symmetric in $a$ and $b$. Since the Laplacian is a symmetric matrix, we observe that
\[\e_a^T LM \e_a=\e_a^T L\ob{\e_a - \e_b}=\e_a^T L \e_a-\e_a^T L \e_b=\e_a^T L \e_a-\e_b^T L \e_a=\e_a^T ML \e_a.\]
Further, for any $k\notin\cb{a,b}$, we have $\e_k^T LM \e_k=0=\e_k^T ML \e_k.$ Now the degrees of $u_a$ and $u_b$ are equal, and therefore $\e_a^T L\e_a=\e_b^T L\e_b$. Hence we have
\begin{eqnarray*}
\e_a^T LM \e_b = \e_a^T L\ob{-e_a + \e_b}=-\e_a^T L e_a + \e_a^T L\e_b=\ob{\e_a^T - \e_b^T} L\e_b= \e_a^T ML \e_b.
\end{eqnarray*}
For an arbitrary $k\notin \cb{a,b}$, notice that $\e_a^T LM \e_k=0$. Since $N\ob{u_a}\setminus\cb{u_b}=N\ob{u_b}\setminus\cb{u_a},$ we have
\[\e_a^T ML \e_k = \ob{\e_a^T - \e_b^T} L\e_k=0.\]
This completes the proof.
\end{proof}
Now we determine the transition matrix of the edge perturbed graph as follows. Recall that $M=\ob{\e_a-\e_b}\ob{\e_a-\e_b}^T$ and therefore
\[M^2=\ob{\e_a-\e_b}\ob{\e_a-\e_b}^T\ob{\e_a-\e_b}\ob{\e_a-\e_b}^T=2\ob{\e_a-\e_b}\ob{\e_a-\e_b}^T=2M.\]
In general, we obtain $M^k=2^{k-1}M,$ whenever $k\geq 1.$ Observe that
\begin{eqnarray*}
\exp{\ob{-i\alpha tM}} &=& I+\sum\limits_{k\geq 1} \frac{\ob{-i\alpha t}^k}{k!}M^k\\
&=& I+\sum\limits_{k\geq 1} \frac{\ob{-i\alpha t}^k}{k!}2^{k-1}M\\
&=& I+\frac{1}{2}\ob{\exp{\ob{-2i\alpha t}}-1}M.
\end{eqnarray*}
If $L$ and $M$ commute then
\begin{eqnarray*}
\exp{\ob{-itL^{\alpha}}}=\exp{\ob{-it\ob{L+\alpha M}}} &=& \exp{\ob{-itL}}\exp{\ob{-i\alpha tM}}\\
&=&\exp{\ob{-itL}}\tb{I+\frac{1}{2}\ob{\exp{\ob{-2i\alpha t}}-1}M},
\end{eqnarray*}
which determines the transition matrix of the perturbed graph in terms of the transition matrix of the unperturbed graph. We incorporate the above observation as follows.
\begin{prop}\label{tm}
Let $G$ be a graph on $n$ vertices $u_1,u_2,\ldots,u_n$, having Laplacian matrix $L$. Suppose that $u_a$ and $u_b$ are two distinct vertices of $G$ with $N\ob{u_a}\setminus\cb{u_b}=N\ob{u_b}\setminus\cb{u_a}$. Then the transition matrix of the edge perturbed graph with Laplacian $L^{\alpha}=L+\alpha M$, where $M=\ob{\e_a-\e_b}\ob{\e_a-\e_b}^T$ and $\alpha\in\Rl$, is given by
\[U_{L^{\alpha}}(t)=U_{L}(t)\tb{I+\frac{1}{2}\ob{\exp{\ob{-2i\alpha t}}-1}M},\]
where $U_{L}(t)$ is the transition matrix of the unperturbed graph.
\end{prop}

Now we draw the following conclusions. The last part of the proof of Theorem \ref{tm2} uses a technique appears in \cite{god1}.
\begin{thm}\label{tm2}
Suppose the conditions of Proposition \ref{tm} is satisfied. If the unperturbed graph $G$ exhibits Laplacian perfect state transfer at time $\tau$ between the vertices $u_p$ and $u_q$, then so does the edge perturbed graph with Laplacian $L^{\alpha}$ provided one of the following holds:
\begin{enumerate}
\item $p,q\in\cb{a,b}$ with $\alpha\tau\in \pi\Zl$,\label{it:a}
\item $p,q\not\in\cb{a,b}$.\label{it:b}
\end{enumerate}
Moreover, if $p\in\cb{a,b}$ and $q\not\in\cb{a,b}$, then there exists no Laplacian perfect state transfer in the perturbed graph between $u_p$ and $u_q$.
\end{thm}

\begin{proof}
By Proposition \ref{tm}, the transition matrix of the edge perturbed graph is
\[U_{L^{\alpha}}(t)=U_{L}(t)\tb{I+\frac{1}{2}\ob{\exp{\ob{-2i\alpha t}}-1}M},\]
where $M=\ob{\e_a-\e_b}\ob{\e_a-\e_b}^T$ and $\alpha\in\Rl$. Note that $\alpha\tau\in \pi\Zl$ implies $\exp{\ob{-2i\alpha \tau}}=1$, and therefore $U_{L^{\alpha}}(\tau)=U_{L}(\tau)$, which proves part (\ref{it:a}). In case $q\not\in\cb{a,b}$, we have $M\e_q=0$ and hence
\begin{eqnarray}\label{eq2}
U_{L^{\alpha}}(t)\e_q=U_{L}(t)\e_q.
\end{eqnarray}
This proves our claim in part (\ref{it:b}). Since $N\ob{u_a}\setminus\cb{u_b}=N\ob{u_b}\setminus\cb{u_a},$ there is an automorphism of $G$ swapping the vertices $u_a$ and $u_b$, and that fixing all other vertices. Suppose $P$ is the matrix of the automorphism then
\[P\e_a=\e_b \text{ and } P\e_q=\e_q,~q\not\in\cb{a,b}.\]
In this case $P$ commutes with the Laplacian matrix $L$. As $U_{L}(t)$ is a polynomial in $L$, the matrix $P$ commutes with $U_{L}(t)$ as well. Now using \eqref{eq2}, we conclude that
\[\e_a^T U_{L^{\alpha}}(t)\e_q=\e_a^T U_{L}(t)\e_q=\e_a^T U_{L}(t)P\e_q=\e_a^T PU_{L}(t)\e_q=\e_b^T U_{L^{\alpha}}(t)\e_q,~q\not\in\cb{a,b}.\]
Since $U_{L^{\alpha}}(t)$ is an unitary matrix and $a\neq b$, there is no LPST in the perturbed graph between $u_p$ and $u_q$ whenever $p\in\cb{a,b}$ and $q\not\in\cb{a,b}$.
\end{proof}

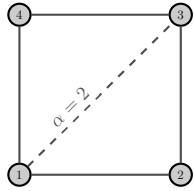
\begin{figure}[h]
\centering
\begin{tikzpicture}[scale=.3,auto=left, rotate=-45]
\tikzstyle{every node}=[circle, thick, fill=black!20, scale=0.4]

\def \r {5cm} 

 \node[draw] (1) at (-90+360*0/4:\r) {$1$};
 \node[draw] (2) at (-90+360*1/4:\r) {$2$};
 \node[draw] (3) at (-90+360*2/4:\r) {$3$};
 \node[draw] (4) at (-90+360*3/4:\r) {$4$};

  \draw[thick,black!70] (1)--(2)--(3)--(4)--(1);
  \draw[thick,dashed,black!70] (1) -- (3) node[midway, rectangle, fill=white, rotate=45, thick, scale=1.5] {$\alpha=2$};
 \end{tikzpicture}
\caption{The edge perturbed cycle $C_4$.}
\label{fg1}
\end{figure}

We know that the cycle $C_4$ exhibits PST at $\tau=\frac{\pi}{2}$ between every pair of antipodal vertices (see \cite{god1}). If we add an edge between such a pair with weight $\alpha=2$ (see Figure \ref{fg1}) then $\alpha\tau\in \pi\Zl$. Now Theorem \ref{tm2} implies that the perturbed graph also exhibits PST at $\frac{\pi}{2}$ between the same pair of vertices as $C_4$. Later we shall also use Theorem \ref{tm2} to find LPST in unweighted graphs. Next we register another conclusion deduced from Proposition \ref{tm}.

\begin{thm}\label{tm3}
Suppose the conditions of Proposition \ref{tm} is satisfied. Let the unperturbed graph is periodic at $u_p$ at time $\tau$. Then the following holds:
\begin{enumerate}
\item If $p\in\cb{a,b}$ with $2\alpha\tau\in \pi(2\Zl+1)$, then the edge perturbed graph exhibits Laplacian perfect state transfer between $u_a$ and $u_b$ at $\tau$.\label{it:c}
\item If $p\not\in\cb{a,b}$, then the edge perturbed graph is also periodic at the vertex $u_p$ at $\tau$.\label{it:d}
\end{enumerate}
 \end{thm}
\begin{proof}
Using Proposition \ref{tm}, we obtain
\[U_{L^{\alpha}}(t)=U_{L}(t)\tb{I+\frac{1}{2}\ob{\exp{\ob{-2i\alpha t}}-1}M},\]
where $M=\ob{\e_a-\e_b}\ob{\e_a-\e_b}^T$ and $\alpha\in\Rl$. If it satisfies $2\alpha\tau\in \pi(2\Zl+1)$ then $\exp{\ob{-2i\alpha \tau}}=-1$, and therefore
\[U_{L^{\alpha}}(\tau)=U_{L}(\tau)\ob{I-M}.\]
In part (\ref{it:c}), without loss of generality let $p=a$. Hence the following simplifies to
\[U_{L^{\alpha}}(\tau)\e_b=U_{L}(\tau)\ob{I-M} \e_b=U_{L}(\tau) \e_a.\]
If, in addition, the unperturbed graph is periodic at $u_a$ then it can be observed that the edge perturbed graph exhibits LPST between $u_a$ and $u_b$ at $\tau$. Finally, if $p\notin\cb{a,b}$, then we have $M\e_p=0$, and
\[U_{L^{\alpha}}(\tau)\e_p=U_{L}(\tau) \e_p.\]
Hence the claim in part (\ref{it:d}) is now follows.
\end{proof}

Consider the following corollary which generalizes the main result in \cite{bose1} on Laplacian perfect state transfer in complete graphs with a missing edge.

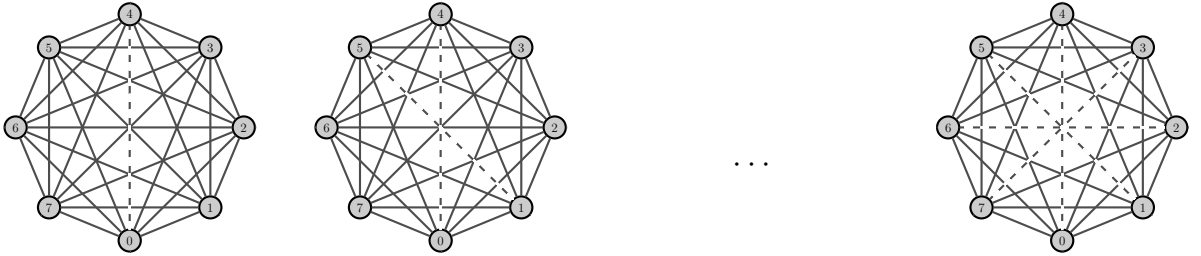
\begin{figure}[h]
\centering
\begin{multicols}{4}
\begin{tikzpicture}[scale=.3,auto=left]
\tikzstyle{every node}=[draw, circle, thick, fill=black!20, scale=0.4]

\def \r {5cm} 

 \node (0) at (-90+360*0/8:\r) {$0$};
 \node (1) at (-90+360*1/8:\r) {$1$};
 \node (2) at (-90+360*2/8:\r) {$2$};
 \node (3) at (-90+360*3/8:\r) {$3$};
 \node (4) at (-90+360*4/8:\r) {$4$};
 \node (5) at (-90+360*5/8:\r) {$5$};
 \node (6) at (-90+360*6/8:\r) {$6$};
 \node (7) at (-90+360*7/8:\r) {$7$};

\foreach \x in {0,...,7}{ 
 \foreach \y in {\x,...,7}{
  \draw[thick,black!70] (\x)--(\y);
  \draw[thick,dashed,white] (0) -- (4);

 }
 }
\end{tikzpicture}

\columnbreak

\begin{tikzpicture}[scale=.3,auto=left]
\tikzstyle{every node}=[draw, circle, thick, fill=black!20, scale=0.4]

\def \r {5cm} 

 \node (0) at (-90+360*0/8:\r) {$0$};
 \node (1) at (-90+360*1/8:\r) {$1$};
 \node (2) at (-90+360*2/8:\r) {$2$};
 \node (3) at (-90+360*3/8:\r) {$3$};
 \node (4) at (-90+360*4/8:\r) {$4$};
 \node (5) at (-90+360*5/8:\r) {$5$};
 \node (6) at (-90+360*6/8:\r) {$6$};
 \node (7) at (-90+360*7/8:\r) {$7$};

\foreach \x in {0,...,7}{ 
 \foreach \y in {\x,...,7}{
  \draw[thick,black!70] (\x)--(\y);
  \draw[thick,dashed,white] (0) -- (4);
  \draw[thick,dashed,white] (1) -- (5);
 }
 }
\end{tikzpicture}

\columnbreak

\[\cdots\]

\columnbreak
\begin{tikzpicture}[scale=.3,auto=left]
\tikzstyle{every node}=[draw, circle, thick, fill=black!20, scale=0.4]

\def \r {5cm} 

 \node (0) at (-90+360*0/8:\r) {$0$};
 \node (1) at (-90+360*1/8:\r) {$1$};
 \node (2) at (-90+360*2/8:\r) {$2$};
 \node (3) at (-90+360*3/8:\r) {$3$};
 \node (4) at (-90+360*4/8:\r) {$4$};
 \node (5) at (-90+360*5/8:\r) {$5$};
 \node (6) at (-90+360*6/8:\r) {$6$};
 \node (7) at (-90+360*7/8:\r) {$7$};

\foreach \x in {0,...,7}{ 
 \foreach \y in {\x,...,7}{
  \draw[thick,black!70] (\x)--(\y);
  \draw[thick,dashed,white] (0) -- (4);
  \draw[thick,dashed,white] (1) -- (5);
  \draw[thick,dashed,white] (2) -- (6);
  \draw[thick,dashed,white] (3) -- (7);
 }
 }
\end{tikzpicture}

\end{multicols}
\caption{The complete graph $K_8$ with disjoint edges removed.}
\label{fg2}
\end{figure}

\begin{cor}\label{cor1}
 The complete graph $K_{4n}$ on $4n$ vertices with a missing edge exhibits Laplacian perfect state transfer. Moreover, removal of any set of pairwise non-adjacent edges from $K_{4n}$ results Laplacian perfect state transfer at $\frac{\pi}{2}$ between the end vertices of every edge removed.
\end{cor}
\begin{proof}
The transition matrix of $K_{4n}$, as given in \eqref{eq0}, is obtained by 
\[U_{L}(t)=\exp{\ob{-it L}}=\frac{1}{n}\J +\exp{\ob{-4int}}\ob{I-\frac{1}{n}\J}.\]
If an edge between two vertices $u$ and $v$ is deleted then the perturbed graph has the Laplacian $L^{\alpha}$ with $\alpha=-1$. Note that $K_{4n}$ is periodic with $\tau=\frac{\pi}{2}$. Since $K_{4n}$ is periodic at $u$ with $2\alpha\tau\in \pi(2\Zl+1)$, by Theorem \ref{tm3}, the edge deleted complete graph exhibits LPST between $u$ and $v$ at time $\tau$. Moreover, if $x$ and $y$ are two vertices distinct from both $u$ and $v$ then the perturbed graph is periodic at both $x$ and $y$ at $\tau$. In a succession, if the edge between $x$ and $y$ is now removed then again by Theorem \ref{tm3}, the resulting graph exhibits LPST between $x$ and $y$ at time $\tau$. Also Theorem \ref{tm2} implies that the deletion of the edge between $x$ and $y$ does not disturb LPST between $u$ and $v$ at time $\tau$. Continuing this process, we obtain the desired result as demonstrated in Figure \ref{fg2}.
\end{proof}

Recall that all Laplacian integral graphs are periodic at $\tau=2\pi.$ Assume that $G$ is a Laplacian integral graph having a pair of vertices $u$ and $v$ with $N\ob{u}\setminus\cb{v}=N\ob{v}\setminus\cb{u}$. Suppose we reset the edge weight between $u$ and $v$ to $\frac{1}{4}$. It means that, if the unperturbed graph has an edge between $u$ and $v$ then consider $\alpha=-\frac{3}{4}$, otherwise let $\alpha=\frac{1}{4}$. In any case, we have $2\alpha\tau\in \pi(2\Zl+1)$ and hence Theorem \ref{tm3} applies to have the following conclusion.

\begin{cor}
Suppose $G$ is a Laplacian integral graph having a pair of distinct vertices $u$ and $v$ with $N\ob{u}\setminus\cb{v}=N\ob{v}\setminus\cb{u}$. If the edge weight between $u$ and $v$ is set to $\frac{1}{4}$ then the edge perturbed graph exhibits Laplacian perfect state transfer between $u$ and $v$ at time $2\pi$. Moreover $G$ is periodic at rest of the vertices at $2\pi$.
\end{cor}

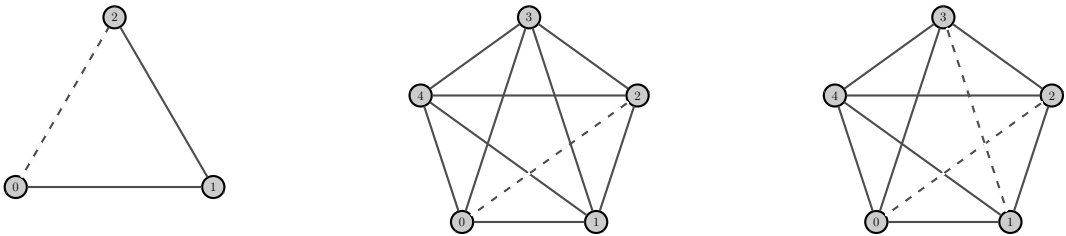
\begin{figure}[h]
\centering
\begin{multicols}{3}
\begin{tikzpicture}[scale=.3,auto=left, rotate=-60]
\tikzstyle{every node}=[draw, circle, thick, fill=black!20, scale=0.4]

\def \r {5cm} 

 \node (0) at (-90+360*0/3:\r) {$0$};
 \node (1) at (-90+360*1/3:\r) {$1$};
 \node (2) at (-90+360*2/3:\r) {$2$};

\foreach \x in {0,...,2}{ 
 \foreach \y in {\x,...,2}{
  \draw[thick,black!70] (\x)--(\y);
  \draw[thick,dashed,white] (0) -- (2);
 }
 }
\end{tikzpicture}

\columnbreak

\begin{tikzpicture}[scale=.3,auto=left, rotate=-36]
\tikzstyle{every node}=[draw, circle, thick, fill=black!20, scale=0.4]

\def \r {5cm} 

 \node (0) at (-90+360*0/5:\r) {$0$};
 \node (1) at (-90+360*1/5:\r) {$1$};
 \node (2) at (-90+360*2/5:\r) {$2$};
 \node (3) at (-90+360*3/5:\r) {$3$};
 \node (4) at (-90+360*4/5:\r) {$4$};

\foreach \x in {0,...,4}{ 
 \foreach \y in {\x,...,4}{
  \draw[thick,black!70] (\x)--(\y);
  \draw[thick,dashed,white] (0) -- (2);

 }
 }
\end{tikzpicture}

\columnbreak

\begin{tikzpicture}[scale=.3,auto=left,  rotate=-36]
\tikzstyle{every node}=[draw, circle, thick, fill=black!20, scale=0.4]

\def \r {5cm} 

 \node (0) at (-90+360*0/5:\r) {$0$};
 \node (1) at (-90+360*1/5:\r) {$1$};
 \node (2) at (-90+360*2/5:\r) {$2$};
 \node (3) at (-90+360*3/5:\r) {$3$};
 \node (4) at (-90+360*4/5:\r) {$4$};

\foreach \x in {0,...,4}{ 
 \foreach \y in {\x,...,4}{
  \draw[thick,black!70] (\x)--(\y);
  \draw[thick,dashed,white] (0) -- (2);
\draw[thick,dashed,white] (1) -- (3);
 }
 }
\end{tikzpicture}

\end{multicols}
\caption{The complete graphs $K_3$ and $K_5$ with disjoint edges perturbed.}
\label{fg3}
\end{figure}

It immediately follows that all non-trivial complete graphs exhibit LPST with a single edge perturbation. Further, repeated use of Theorem \ref{tm2} and Theorem \ref{tm3} imply that perturbing disjoint edges in a complete graph results LPST at $2\pi$ between the end vertices of each perturbed edge. In Figure \ref{fg3}, weights of each dashed edge in the complete graphs $K_3$ and $K_5$ is set to $\frac{1}{4}$ and LPST occurs between the end vertices of each dashed edge.

\section{LPGST on Edge Perturbed Graphs}

In the previous section, we have considered Laplacian state transfer between twin vertices in a graph, and found conditions under which these graphs exhibit perfect state transfer. The conclusions can further be generalized to have LPGST in certain graphs. The following inference are immediate from the proof of Theorem \ref{tm2}, and therefore we omit the proof for convenience.
\begin{thm}\label{tm4}
Suppose the conditions of Proposition \ref{tm} is satisfied. If the unperturbed graph $G$ exhibits Laplacian pretty good state transfer between the vertices $u_p$ and $u_q$ with respect to the sequence $\tau_k\in\Rl$, then so does the edge perturbed graph with Laplacian $L^{\alpha}$ provided one of the following holds:
\begin{enumerate}
\item $p,q\in\cb{a,b}$ with $\alpha\tau_k\in \pi\Zl$,
\item $p,q\not\in\cb{a,b}$.
\end{enumerate}
Moreover, if $p\in\cb{a,b}$ and $q\not\in\cb{a,b}$, then there exists no Laplacian pretty good state transfer in the perturbed graph between $u_p$ and $u_q$.
\end{thm}

Next we include another conclusion which can be deduced immediately from the proof of Theorem \ref{tm3}. 

\begin{thm}\label{tm5}
Suppose the conditions of Proposition \ref{tm} is satisfied. Let the unperturbed graph is almost periodic at $u_p$ with respect to the sequence $\tau_k\in\Rl$. Then the following holds:
\begin{enumerate}
\item If $p\in\cb{a,b}$ with $2\alpha\tau_k\in \pi(2\Zl+1)$, then the edge perturbed graph exhibits Laplacian pretty good state transfer between $u_a$ and $u_b$ with respect to $\tau_k$.
\item If $p\not\in\cb{a,b}$, then the edge perturbed graph is almost periodic at the vertex $u_p$ with respect to $\tau_k$.
\end{enumerate}
 \end{thm}

Now we use Theorem \ref{tm4} and Theorem \ref{tm5} to find a class of edge perturbed circulant graphs exhibiting LPGST. Before that we revisit the proof of Theorem 2.8 in \cite{pal7} to establish the following result that plays a crucial role in classifying LPGST in circulant graphs. We use the resources as placed in \cite{pal7} and avoid unnecessary repetitions to prove the result. It is noteworthy that the continuous-time quantum walk relative to adjacency matrix has been considered in the following.

\begin{thm}\label{tm6}
Let $k\in\Nl$, $n=2^k$ and consider the circulant graph $Cay\left(\Zl_n,S\right).$ If each divisor $d$ of $n$ satisfies $\left|S\cap S_n(d)\right|\equiv 0\pmod{4}$ then $Cay\left(\Zl_n,S\right)$ is almost periodic with respect to a sequence in $\left(4\Zl+1\right)\frac{\pi}{2}$.
\end{thm}

\begin{proof}
Recall that the eigenvalues of the cycle $C_n$ are 
\[\lambda_l = 2\cos{\left(\frac{2l\pi}{n}\right)},~l=0,1,\ldots, n-1.\]
Now the distinct positive eigenvalues of $C_n$ are linearly independent over $\Ql$ as appears in the proof of Theorem 2.8 in \cite{pal7}. For $ 1\leq l\leq \frac{n}{4}-1$, let us choose $\alpha_l=-\frac{\lambda_l}{4}$. By the Kronecker approximation theorem, for $\delta>0$ there exist $q,m_1,\ldots,m_{\frac{n}{4}-1}\in\Zl$ such that for $l=1,\ldots, \frac{n}{4}-1$
\begin{eqnarray}\label{T4E1}
\left|q\lambda_l-m_l-\alpha_l\right|<\frac{\delta}{2n\pi},\; \emph{i.e,}\;\left|\left((4q+1)\frac{\pi}{2}\right)\lambda_l-2m_l\pi\right|<\frac{\delta}{n}.
\end{eqnarray} 
Note that $\lambda_l=-\lambda_{\frac{n}{2}-l}=-\lambda_{\frac{n}{2}+l}=\lambda_{n-l}$ holds for $ 1\leq l\leq \frac{n}{4}-1$, and 
\[\lambda_0=2,~\lambda_{\frac{n}{4}}=0,~\lambda_{\frac{n}{2}}=-2 \text{ and } \lambda_{\frac{3n}{4}}=0.\]
If $\theta_l$ denotes the eigenvalues of $Cay\left(\Zl_n,S\right)$ then
\[\theta_l=\frac{1}{2}\sum\limits_{s\in S}\lambda_{ls}=\sum\limits_{d\mid n}\frac{1}{2}\tb{\sum\limits_{s\in S\cap S_n(d)}\lambda_{ls}}.\]
Since each divisor $d~\left(\neq \frac{n}{2},\frac{n}{4}\right)$ of $n$ satisfy $\left|S\cap S_n(d)\right|\equiv 0\pmod{4}$, using Equation \ref{T4E1} and the triangle inequality, we obtain that for $\delta>0$ there exists $t\in\frac{\pi}{2}\left(4\Zl+1\right)$ so that for each $l$ there exists an integer $l''$ such that 
\begin{eqnarray*}
\left|\theta_l t- 2l''\pi\right|<\delta.
\end{eqnarray*}
The uniform continuity of exponential function implies that for $\epsilon>0$, there exists $t\in\frac{\pi}{2}\left(4\Zl+1\right)$ such that $\left|\exp{\left[-i\theta_l t\right]}- 1\right|<\epsilon.$ If $H_{S}(t)$ is the transition matrix relative to the adjacency matrix of $Cay\left(\Zl_n,S\right)$ then we obtain the following as given in the proof of Theorem 2.8 in \cite{pal7}.
\begin{eqnarray*}
\left|\left[H_{S}(t)\right]_{0,0}-1\right|=\frac{1}{n}\left|\sum\limits_{l=0}^{n-1}\left(\exp{\left[-i\theta_l t\right]}-1\right)\right|<\epsilon.
\end{eqnarray*}
Hence $Cay\left(\Zl_n,S\right)$ is almost periodic with respect to a sequence in $\frac{\pi}{2}\left(4\Zl+1\right)$ as concluded in the proof of Theorem 2.8 in \cite{pal7}.
\end{proof}
Recall that all circulant graphs are regular, and therefore Theorem \ref{tm6} also holds true when considering continuous-time quantum walk relative to the Laplacian matrix. Now suppose $k\in\Nl,$ $n=2^k$, and consider a circulant graph $Cay\left(\Zl_n,S\right)$. It can be observed that the vertices $0$ and $\frac{n}{2}$ are twin vertices of $Cay\left(\Zl_n,S\right)$ if and only if $S=\frac{n}{2}-S$. In fact, if $S=\frac{n}{2}-S$ then for each $x\in\Zl_n$ the pair of vertices $x$ and $\frac{n}{2}+x$ are twin vertices of $Cay\left(\Zl_n,S\right)$. Therefore the following can be obtained as a corollary of Theorem \ref{tm5}.

\begin{cor}\label{tm5c1}
Let $k\in\Nl$, $n=2^k$ and consider a circulant graph $Cay\left(\Zl_n,S\right).$ Suppose $S=\frac{n}{2}-S$ and each divisor $d$ of $n$ satisfies $\left|S\cap S_n(d)\right|\equiv 0\pmod{4}.$ If a new edge is added between a pair of twin vertices in $Cay\left(\Zl_n,S\right)$ then the resulting graph exhibits Laplacian pretty good state transfer between the end vertices of the newly added edge with respect to a sequence $\tau_k\in\frac{\pi}{2}\left(4\Zl+1\right).$ Moreover, the perturbed graph is almost periodic at the remaining vertices with respect to $\tau_k$.
\end{cor}

\begin{proof}
Since $S=\frac{n}{2}-S$ and $0\notin S$ the vertices $0$ and $\frac{n}{2}$ are not adjacent in $Cay\left(\Zl_n,S\right).$ Without loss of generality, let a new edge is added between the twin vertices $0$ and $\frac{n}{2}$ in $Cay\left(\Zl_n,S\right).$ By Theorem \ref{tm6}, the graph $Cay\left(\Zl_n,S\right)$ is almost periodic with respect to a sequence $\tau_k\in \frac{\pi}{2}\left(4\Zl+1\right).$ Note that the conditions of Theorem \ref{tm5} applies to $Cay\left(\Zl_n,S\right)$ with $u_a=0,~u_b=\frac{n}{2}$ and $\alpha=1$. Hence the perturbed graph admits LPGST between $u_a=0$ and $u_b=\frac{n}{2}$ with respect to $\tau_k$, and it is almost periodic at the remaining vertices with respect to the same sequence $\tau_k$.
\end{proof}

It is now evident from Theorem \ref{tm4} that new edges can be added successively between the twin vertices in $Cay\left(\Zl_n,S\right)$ to obtain more pair of vertices exhibiting LPGST. We illustrate this with the following example.
\begin{figure}[h]
\centering
\begin{multicols}{4}
\begin{tikzpicture}[scale=.3,auto=left]
\tikzstyle{every node}=[draw, circle, thick, fill=black!20, scale=0.4]

\def \r {5cm} 

 \node (0) at (-90+360*0/8:\r) {$0$};
 \node (1) at (-90+360*1/8:\r) {$1$};
 \node (2) at (-90+360*2/8:\r) {$2$};
 \node (3) at (-90+360*3/8:\r) {$3$};
 \node (4) at (-90+360*4/8:\r) {$4$};
 \node (5) at (-90+360*5/8:\r) {$5$};
 \node (6) at (-90+360*6/8:\r) {$6$};
 \node (7) at (-90+360*7/8:\r) {$7$};

 \foreach \x in {1,3,5,7}{
  \draw[thick,black!70] (0)--(\x);
 }

 \foreach \x in {2,4,6}{
  \draw[thick,black!70] (1)--(\x);
 }
 
  \foreach \x in {3,5,7}{
  \draw[thick,black!70] (2)--(\x);
 }

  \foreach \x in {4,6}{
  \draw[thick,black!70] (3)--(\x);
 }

  \foreach \x in {5,7}{
  \draw[thick,black!70] (4)--(\x);
 }

  \draw[thick,black!70] (5)--(6); 
  \draw[thick,black!70] (6)--(7);

  \draw[thick,dashed] (0) -- (4);

\end{tikzpicture}

\columnbreak

\begin{tikzpicture}[scale=.3,auto=left]
\tikzstyle{every node}=[draw, circle, thick, fill=black!20, scale=0.4]

\def \r {5cm} 

 \node (0) at (-90+360*0/8:\r) {$0$};
 \node (1) at (-90+360*1/8:\r) {$1$};
 \node (2) at (-90+360*2/8:\r) {$2$};
 \node (3) at (-90+360*3/8:\r) {$3$};
 \node (4) at (-90+360*4/8:\r) {$4$};
 \node (5) at (-90+360*5/8:\r) {$5$};
 \node (6) at (-90+360*6/8:\r) {$6$};
 \node (7) at (-90+360*7/8:\r) {$7$};

 \foreach \x in {1,3,5,7}{
  \draw[thick,black!70] (0)--(\x);
 }

 \foreach \x in {2,4,6}{
  \draw[thick,black!70] (1)--(\x);
 }
 
  \foreach \x in {3,5,7}{
  \draw[thick,black!70] (2)--(\x);
 }

  \foreach \x in {4,6}{
  \draw[thick,black!70] (3)--(\x);
 }

  \foreach \x in {5,7}{
  \draw[thick,black!70] (4)--(\x);
 }

  \draw[thick,black!70] (5)--(6); 
  \draw[thick,black!70] (6)--(7);

  \draw[thick,dashed] (0) -- (4);
  \draw[thick,dashed] (1) -- (5);

\end{tikzpicture}

\columnbreak

\[\cdots\]

\columnbreak

\begin{tikzpicture}[scale=.3,auto=left]
\tikzstyle{every node}=[draw, circle, thick, fill=black!20, scale=0.4]

\def \r {5cm} 

 \node (0) at (-90+360*0/8:\r) {$0$};
 \node (1) at (-90+360*1/8:\r) {$1$};
 \node (2) at (-90+360*2/8:\r) {$2$};
 \node (3) at (-90+360*3/8:\r) {$3$};
 \node (4) at (-90+360*4/8:\r) {$4$};
 \node (5) at (-90+360*5/8:\r) {$5$};
 \node (6) at (-90+360*6/8:\r) {$6$};
 \node (7) at (-90+360*7/8:\r) {$7$};

 \foreach \x in {1,3,5,7}{
  \draw[thick,black!70] (0)--(\x);
 }

 \foreach \x in {2,4,6}{
  \draw[thick,black!70] (1)--(\x);
 }
 
  \foreach \x in {3,5,7}{
  \draw[thick,black!70] (2)--(\x);
 }

  \foreach \x in {4,6}{
  \draw[thick,black!70] (3)--(\x);
 }

  \foreach \x in {5,7}{
  \draw[thick,black!70] (4)--(\x);
 }

  \draw[thick,black!70] (5)--(6); 
  \draw[thick,black!70] (6)--(7);

  \draw[thick,dashed] (0) -- (4);
  \draw[thick,dashed] (1) -- (5);
  \draw[thick,dashed] (2) -- (6);
  \draw[thick,dashed] (3) -- (7);

\end{tikzpicture}

\end{multicols}
\caption{Edges perturbed circulant graph $Cay\left(\Zl_n,S\right)$ with $S=\cb{1,3,5,7}.$}
\label{fg4}
\end{figure}
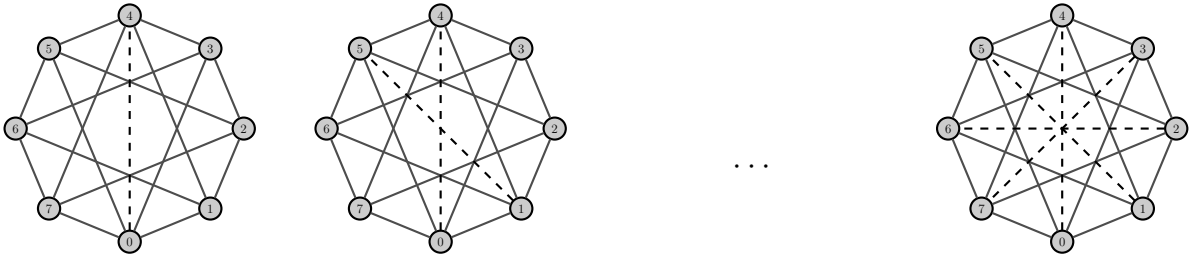

Suppose $n=8$, $S=\cb{1,3,5,7}$ and consider the circulant graph $G=Cay\left(\Zl_n,S\right).$ Note that for $x\in\Zl_n$, the pair $x$ and $x+\frac{n}{2}$ are twin vertices in $G$. Now if the vertices $0$ and $\frac{n}{2}$ are joined by a new edge then by Corollary \ref{tm5c1}, the perturbed graph exhibits LPGST between $0$ and $\frac{n}{2}$ with respect to a sequence $\tau_k\in \frac{\pi}{2}\left(4\Zl+1\right)$, and it is almost periodic at the remaining vertices with respect to the same sequence $\tau_k$. Now if an another edge is added to the perturbed graph between $1$ and $1+\frac{n}{2}$ then by Theorem \ref{tm4}, the resulting graph exhibits LPGST between the pair of vertices $0$ and $\frac{n}{2}$ as well as $1$ and $1+\frac{n}{2}$ (see Figure \ref{fg4}). Likewise new edges can further be added to obtain more pair of vertices exhibiting LPGST. Although all circulant graphs satisfying conditions of Corollary \ref{tm5c1} need not be integral, Theorem \ref{so} ensures that $G$ is an integral graph as $S=S_8(1)$, and hence $G$ is periodic at $2\pi.$ Now $G$ is almost periodic with respect to $\tau_k\in \frac{\pi}{2}\left(4\Zl+1\right)$, which implies that $\tau_k$ can be chosen to be the constant sequence $\tau_k=\frac{\pi}{2}.$ Hence all those perturbed graphs obtained from $G$ exhibit Laplacian perfect state transfer at time $\frac{\pi}{2}.$ We may generalize this to all integral circulant graphs satisfying the conditions of Corollary \ref{tm5c1}.

\begin{cor}\label{tm5c2}
Let $k\in\Nl$, $n=2^k$ and consider an integral circulant graph $Cay\left(\Zl_n,S\right).$ Suppose $S=\frac{n}{2}-S$ and each divisor $d$ of $n$ satisfies $\left|S\cap S_n(d)\right|\equiv 0\pmod{4}.$ If a new edge is added between a pair of twin vertices in $Cay\left(\Zl_n,S\right)$ then the resulting graph exhibits Laplacian perfect state transfer between the end vertices of the newly added edge at $\frac{\pi}{2}.$ Moreover, the perturbed graph is periodic at the remaining vertices at $\frac{\pi}{2}$.
\end{cor}

As in case of the circulant graph $Cay\left(\Zl_n,S\right)$ with $S=\cb{1,3,5,7},$ we may now successively add edges in $Cay\left(\Zl_n,S\right)$, which satisfy the conditions of Corollary \ref{tm5c2}, to have more pair of twin vertices exhibiting LPST at $\frac{\pi}{2}$ in the edge perturbed graph.

%

\bibliographystyle{abbrv}
\bibliography{edge_perturbation}

\begin{thebibliography}{10}

\bibitem{ack}
E.~Ackelsberg, Z.~Brehm, A.~Chan, J.~Mundinger, and C.~Tamon.
\newblock Laplacian state transfer in coronas.
\newblock {\em Linear Algebra Appl.}, 506:154--167, 2016.

\bibitem{alv}
R.~Alvir, S.~Dever, B.~Lovitz, J.~Myer, C.~Tamon, Y.~Xu, and H.~Zhan.
\newblock Perfect state transfer in {L}aplacian quantum walk.
\newblock {\em J. Algebraic Combin.}, 43(4):801--826, 2016.

\bibitem{apo}
T.~M. Apostol.
\newblock {\em Modular functions and {D}irichlet series in number theory},
  volume~41 of {\em Graduate Texts in Mathematics}.
\newblock Springer-Verlag, New York, second edition, 1990.

\bibitem{mil4}
M.~Ba\v{s}i\'{c}.
\newblock Characterization of quantum circulant networks having perfect state
  transfer.
\newblock {\em Quantum Inf. Process.}, 12(1):345--364, 2013.

\bibitem{ber}
A.~Bernasconi, C.~Godsil, and S.~Severini.
\newblock Quantum networks on cubelike graphs.
\newblock {\em Phys. Rev. A (3)}, 78(5):052320, 5, 2008.

\bibitem{bose}
S.~Bose.
\newblock Quantum communication through an unmodulated spin chain.
\newblock {\em Physical review letters}, 91:207901, 2003.

\bibitem{bose1}
S.~{Bose}, A.~{Casaccino}, S.~{Mancini}, and S.~{Severini}.
\newblock {Communication in \(XYZ\) all-to-all quantum networks with a missing
  link}.
\newblock {\em {Int. J. Quantum Inf.}}, 7(4):713--723, 2009.

\bibitem{che}
W.-C. Cheung and C.~Godsil.
\newblock Perfect state transfer in cubelike graphs.
\newblock {\em Linear Algebra Appl.}, 435(10):2468--2474, 2011.

\bibitem{coug}
G.~Coutinho and C.~Godsil.
\newblock Graph spectra and continuous quantum walks.
\newblock {\em (in preparation)}, 2021.

\bibitem{cou2}
G.~Coutinho, C.~Godsil, K.~Guo, and F.~Vanhove.
\newblock Perfect state transfer on distance-regular graphs and association
  schemes.
\newblock {\em Linear Algebra Appl.}, 478:108--130, 2015.

\bibitem{cou0}
G.~Coutinho and H.~Liu.
\newblock No {L}aplacian perfect state transfer in trees.
\newblock {\em SIAM J. Discrete Math.}, 29(4):2179--2188, 2015.

\bibitem{farhi}
E.~Farhi and S.~Gutmann.
\newblock Quantum computation and decision trees.
\newblock {\em Phys. Rev. A (3)}, 58(2):915--928, 1998.

\bibitem{god3}
C.~Godsil.
\newblock Periodic graphs.
\newblock {\em Electron. J. Combin.}, 18(1):Paper 23, 15, 2011.

\bibitem{god1}
C.~Godsil.
\newblock State transfer on graphs.
\newblock {\em Discrete Math.}, 312(1):129--147, 2012.

\bibitem{god0}
C.~Godsil and G.~Royle.
\newblock {\em Algebraic graph theory}, volume 207 of {\em Graduate Texts in
  Mathematics}.
\newblock Springer-Verlag, New York, 2001.

\bibitem{horn}
R.~A. Horn and C.~R. Johnson.
\newblock {\em Matrix analysis}.
\newblock Cambridge University Press, Cambridge, second edition, 2013.

\bibitem{krik1}
S.~{Kirkland} and S.~{Severini}.
\newblock Spin-system dynamics and fault detection in threshold networks.
\newblock {\em Physical Review. A}, 83(1):012310, Jan. 2011.

\bibitem{li}
Y.~Li, X.~Liu, and S.~Zhang.
\newblock Laplacian state transfer in {$Q$}-graph.
\newblock {\em Appl. Math. Comput.}, 384:125370, 11, 2020.

\bibitem{liu}
X.~Liu and Q.~Wang.
\newblock Laplacian state transfer in total graphs.
\newblock {\em Discrete Math.}, 344(1):Paper No. 112139, 11, 2021.

\bibitem{nik}
G.~M. Nikolopoulos and I.~Jex, editors.
\newblock {\em Quantum state transfer and network engineering}.
\newblock Quantum Science and Technology. Springer, Heidelberg, 2014.

\bibitem{pal6}
H.~Pal.
\newblock More circulant graphs exhibiting pretty good state transfer.
\newblock {\em Discrete Math.}, 341(4):889--895, 2018.

\bibitem{pal7}
H.~Pal.
\newblock Quantum state transfer on a class of circulant graphs.
\newblock {\em Linear and Multilinear Algebra}, 0(0):1--12, 2019.

\bibitem{pal4}
H.~Pal and B.~Bhattacharjya.
\newblock Pretty good state transfer on circulant graphs.
\newblock {\em Electron. J. Combin.}, 24(2):Paper No. 2.23, 13, 2017.

\bibitem{so1}
W.~So.
\newblock Rank one perturbation and its application to the {L}aplacian spectrum
  of a graph.
\newblock {\em Linear and Multilinear Algebra}, 46(3):193--198, 1999.

\bibitem{so}
W.~So.
\newblock Integral circulant graphs.
\newblock {\em Discrete Math.}, 306(1):153--158, 2006.

\bibitem{wang}
J.~Wang and X.~Liu.
\newblock Laplacian state transfer in edge complemented coronas.
\newblock {\em Discrete Appl. Math.}, 293:1--14, 2021.

\end{thebibliography}

\end{document}